\documentclass{article}
\usepackage{amssymb,amsmath}
\usepackage{graphicx}
\usepackage{textcomp}

\def\qed{\hfill {\hbox{${\vcenter{\vbox{               
   \hrule height 0.4pt\hbox{\vrule width 0.4pt height 6pt
   \kern5pt\vrule width 0.4pt}\hrule height 0.4pt}}}$}}}

\def\tr{\triangleright}
\def\bar{\overline}

\newtheorem{theorem}{Theorem}
\newtheorem{definition}{Definition}

\newtheorem{corollary}[theorem]{Corollary}
\newtheorem{example}{Example}

\newtheorem{remark}{Remark}

\newenvironment{proof}[1][Proof]{\smallskip\noindent{\bf #1.}\quad}%
{\qed\par\medskip}

\date{}

\author{\begin{tabular}{c}Jackson Blankstein \\
\footnotesize{jacksblanks@gmail.com}\end{tabular}
\and \begin{tabular}{c}Susan Kim \\ \footnotesize{snooziesoosie@gmail.com}\end{tabular}
\and \begin{tabular}{c}Catherine Lepel\\ \footnotesize{clepel90@gmail.com}\end{tabular}
\and \begin{tabular}{c}Sam Nelson \\ \footnotesize{knots@esotericka.org}\end{tabular}
\and \begin{tabular}{c}Nicole Sanderson \\ \footnotesize{nfsanderson@ucdavis.edu}\end{tabular}\footnote{This paper is the end result of work done at the Claremont Colleges Mathematics
REU Site with funding from the NSF (Award \# DMS 0755540).}}

\title{\Large \textbf{Virtual shadow modules and their link invariants} }

\begin{document}
\maketitle

\begin{abstract}
We introduce an algebra $\mathbb{Z}[X,S]$ associated to a pair $X,S$ of a 
virtual birack $X$ and $X$-shadow $S$. We use modules over $\mathbb{Z}[X,S]$ 
to define enhancements of the virtual birack shadow counting invariant, 
extending the birack shadow module invariants to virtual case. We repeat 
this construction for the twisted virtual case. As applications, we show 
that the new invariants can detect orientation reversal and are not 
determined by the knot group, the Arrow polynomial and the Miyazawa 
polynomial, and that the twisted version is not determined
by the twisted Jones polynomial.
\end{abstract}

\medskip
\quad
\parbox{5.5in}{
\smallskip

}

\section{\large\textbf{Introduction}}

Virtual knots and links were introduced in \cite{K} and have been the subject
of much study since. Twisted virtual knots and links were introduced in 
\cite{B} and have been studied in papers such as \cite{CN6,K2,SPC}. Biracks 
(including biquandles) were first introduced in 
\cite{FRS} as an algebraic structure defining invariants of framed knots and 
links in $S^3$. Biquandle-based representational invariants of virtual
knots and links were studied in papers such as \cite{CYB,FJK,NV}. Biquandles 
were generalized to \textit{virtual biquandles}, structures involving 
operations at virtual crossings as well as classical crossings, in \cite{KM}. 
In \cite{N} a representational invariant of unframed classical and virtual 
knots and links, the \textit{integral birack counting invariant} 
$\Phi_{X}^{\mathbb{Z}}(L)$, was defined. Enhancements of 
$\Phi_{X}^{\mathbb{Z}}(L)$, i.e., invariants which specialize to 
$\Phi_{X}^{\mathbb{Z}}(L)$ but are generally stronger invariants, 
have been studied in various papers such as \cite{BN,CN}.

In \cite{AG} an associative algebra known as the \textit{rack algebra} was
defined from a finite rack $X$. In \cite{CEGS} quandle algebras were used
to enhance the quandle counting invariant, and later in \cite{HHNYZ} a 
modified rack algebra was used to enhance the rack counting invariant.
In \cite{BN}, rack algebras were generalized to \textit{birack algebras}.
In \cite{NP}, birack algebras were further generalized to 
\textit{birack shadow algebras} associated to pairs $X,S$ where $X$ is 
a birack and $S$ is an \textit{$X$-shadow}, i.e. a set with an $X$-action 
satisfying certain diagrammatically motivated properties.
Birack shadow algebra invariants as defined in \cite{NP} are well defined
for classical knots and links but not for virtual knots and links.

In this paper we introduce \textit{virtual birack shadow algebras} and 
\textit{twisted virtual birack shadow algebras}. As an application, we 
use modules over these algebras to enhance the virtual birack and twisted 
virtual birack counting invariants, and we show that the enhanced invariants
can detect orientation reversal and are not determined by the knot group,
the Arrow or Miyazawa polynomials, or the twisted Jones polynomial.

The paper is organized as follows. In section \ref{bb} we recall the basics
of virtual knots, virtual biracks and virtual birack shadows, then introduce 
virtual birack shadow algebras and modules and use these to define an infinite
family of enhanced virtual link invariants. In section \ref{tvbs} we 
introduce twisted virtual birack shadow algebras and modules and use 
these to define an infinite family of enhanced twisted virtual link invariants. 
We collect a few computations and examples in section \ref{cex}, and we 
conclude in section \ref{q} with a few open questions for further research.


\section{\large\textbf{Virtual birack shadow algebras}}\label{bb}

\subsection{Virtual knots and links}

\textit{Virtual knots and links} were introduced in \cite{K} by Louis Kauffman 
in 1996 as a combinatorial generalization of classical knots and links. 
Every oriented knot or link diagram is a planar 4-valent directed graph with 
crossing information at the vertices. Edges in this graph are called 
\textit{semiarcs}. Such a graph can be encoded as a 
\textit{signed Gauss code} by naming the crossings, choosing a base point
on each component, and then noting the order in which the over and under 
instances of each crossing are encountered when following the orientation
of each component. 
\[\includegraphics{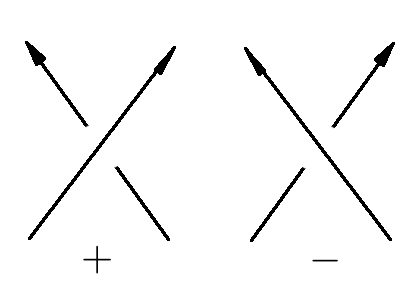} \quad 
\includegraphics{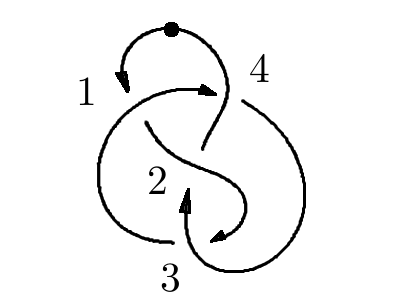} 
\raisebox{0.5in}{$U1^-O2^+U3^+O1^-U4^-O3^+U2^+O4^-$}
\]

A \textit{virtual link} is then an equivalence class of
signed Gauss codes under the Gauss code versions of the Reidemeister moves.
Attempting to reconstruct the original oriented link diagram from a signed
Gauss code, one quickly finds that some signed Gauss codes determine 
nonplanar diagrams, i.e. diagrams needing extra crossings not listed in
the Gauss code. An equivalence class of signed Gauss codes is a 
\textit{classical link} if it contains a signed Gauss code corresponding to a
planar oriented link diagram.

For the nonplanar diagrams, Kauffman introduced \textit{virtual crossings}
drawn as circled self-intersections which do not appear in Gauss codes
along with the rule that two virtual link diagrams are equivalent if their 
signed Gauss codes are equivalent by Gauss code Reidemeister moves. In terms 
of virtual knot diagrams, this breaks down into the \textit{virtual 
Reidemeister moves:}

\[
\begin{array}{cccc}
\multicolumn{4}{c}{
\includegraphics{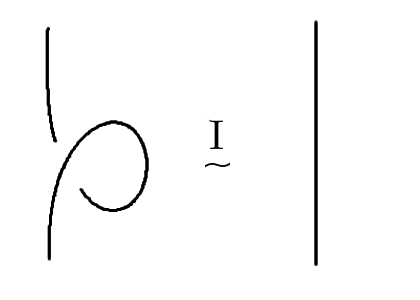} \quad
\includegraphics{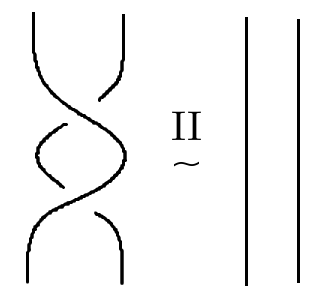} \quad
\includegraphics{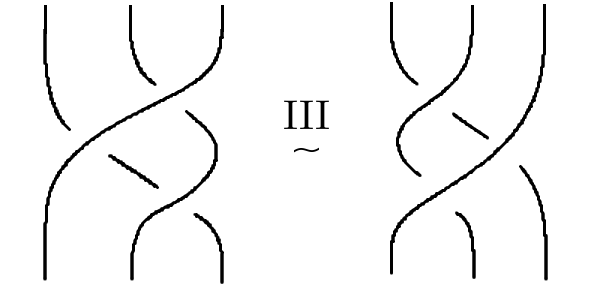}} \\
\includegraphics{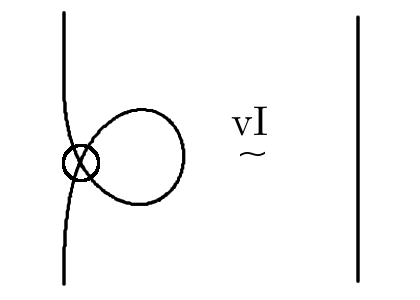} &
\includegraphics{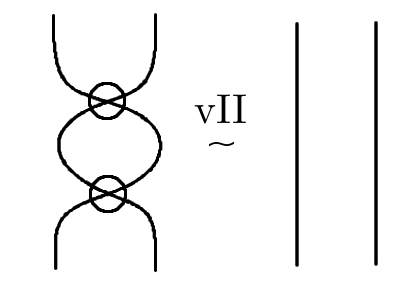} &
\includegraphics{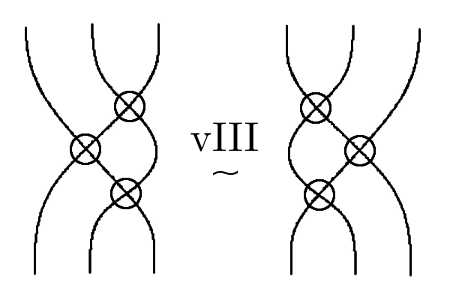} &
\includegraphics{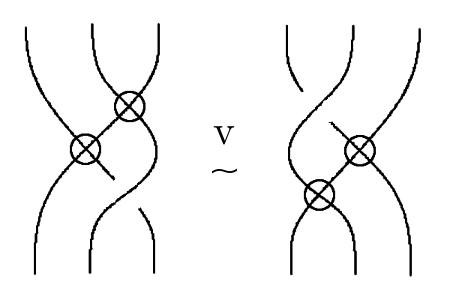} \\
\end{array}
\]

Geometrically, a virtual crossing can be understood as representing genus
in the surface on which the link diagram is drawn; thus, virtual knot theory
is closely related to the theory of knots and links in $I$-bundles over
compact surfaces \cite{KK}.
\[\includegraphics{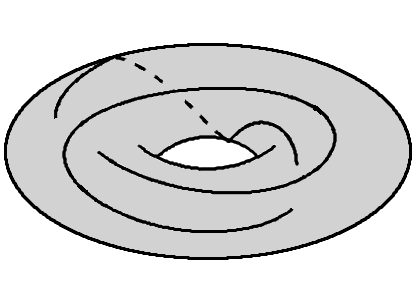} \quad 
\includegraphics{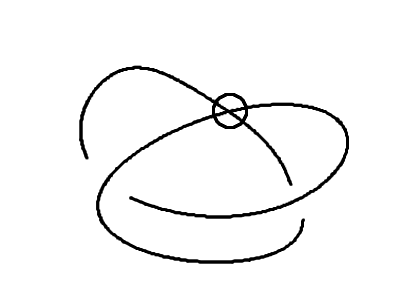}\]
Alternative interpretations of virtual knot diagrams exist, such as Dror 
Bar-Natan's circuit board analogy in which classical crossings represent logic
gates on a circuit board and virtual crossings represent the connections
between the gates (where we don't care which wire goes over or under).

Including virtual crossings restores the planarity of our knot diagram. The 
edges in a virtual knot diagram considered as a planar 4-valent graph with
crossing information (classical or virtual) specified at each vertex will
still be called \textit{semiarcs}; we also define a \textit{classical semiarc}
to be the result of dividing our virtual knot diagram only at classical 
crossing points, i.e. edges in the original nonplanar graph.


\subsection{Virtual Biracks}

A \textit{virtual birack} is an algebraic structure whose axioms encode
the \textit{blackboard framed virtual isotopy moves}, obtained from the
virtual isotopy moves by replacing the usual Reidemeister type I move with
the blackboard framed Reidemeister type I move:
\[
\includegraphics{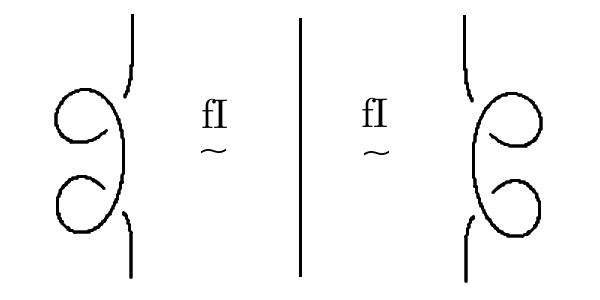} \]

\begin{definition}
\textup{Let $X$ be a set and define $\Delta:X\to X\times X$ by 
$\Delta(x) = (x, x)$. A \textit{virtual birack} structure on $X$ is a pair
of invertible maps $B,V:X \times X  \rightarrow X \times X$ satisfying the 
conditions
\begin{itemize}
\item[(i)]
$B$ and $V$ are  \textit{sideways invertible}, that is there exist unique 
invertible maps $S, vS : X \times X \rightarrow X \times X$ satisfying for all 
$x,y \in X$
\[S(B_1(x,y),x)=(B_2(x,y),y)\quad\mathrm{and}\quad 
vS(V_1(x,y),x)=(V_2(x,y),y),\]
\item[(ii)]
$B$ is \textit{diagonally invertible}, that is, the components 
$(S^{\pm 1} \circ \Delta)_1:X\to X$ and  $(S^{\pm 1}\circ \Delta)_2:X\to X$ 
of the compositions $S\circ \Delta$ and $S^{-1} \circ \Delta$ are bijections,
\item[(iii)] $V$ is \textit{self-inverse} and \textit{diagonal fixing}, 
that is, $V^2=\mathrm{Id}_{X\times X}$ and
$(vS\circ \Delta)_1 = (vS \circ \Delta)_2$;
\item[(iv)]
$B,V$  satisfy the \textit{set-theoretic Yang-Baxter equations:}
\begin{eqnarray*}
(B \times \mathrm{Id}_X)(\mathrm{Id}_X \times B)(B \times \mathrm{Id}_X) & = & 
(\mathrm{Id}_X \times B)(B \times \mathrm{Id}_X)(\mathrm{Id}_X\times B),\\
(V \times \mathrm{Id}_X)(\mathrm{Id}_X \times V)(V \times \mathrm{Id}_X) & = & 
(\mathrm{Id}_X \times V)(V \times \mathrm{Id}_X)(\mathrm{Id}_X \times V), 
\quad \mathrm{and} \\
(B \times \mathrm{Id}_X)(\mathrm{Id}_X \times V)(V \times \mathrm{Id}_X) & = &  
(\mathrm{Id}_X\times V)(V \times \mathrm{Id}_X)(\mathrm{Id}_X \times B).\\
\end{eqnarray*}
\end{itemize}
We will find it convenient to abbreviate $B_1(x,y)=y^x,\ B_2(x,y)=x_y,
V_1(x,y)=y^{\widetilde{x}},$ and $V_2(x,y)=x_{\widetilde{y}}$.}
\end{definition}

We also have
\begin{definition}\textup{
Let $X$ and $Y$ be sets with virtual birack maps $B_X,V_X$ and $B_Y,V_Y$. A
\textit{virtual birack homomorphism} is a map $f:X\to Y$ satisfying}
\[B_Y\circ (f\times f)=(f\times f)\circ B_X \quad \mathrm{and}\quad
V_Y\circ (f\times f)=(f\times f)\circ V_X. \]
\end{definition}

\begin{example}
\textup{Examples of virtual birack structures include:
\begin{itemize}
\item \textit{Racks.} A set $X$ with a self-distributive right-invertible
right action $\tr:X\times X\to X$ is a virtual birack under the maps 
$B(x,y)=(y\tr x,x)$ and $V(x,y)=(y,x)$.
\item \textit{$(v,t,s,r)$-Biracks.} Let 
$\check\Lambda=\mathbb{Z}[v^{\pm 1},t^{\pm 1},s,r^{\pm 1}]/(s^2-(1-tr)s)$. Then 
any $\check\Lambda$-module $X$ is a virtual birack under the operations 
$B(x,y)=(ty+sx,rx)$ and $V(x,y)=(vy,v^{-1}x)$.
\item \textit{Constant action virtual biracks.} Let $X$ be any set and let
$\sigma,\tau,\nu :X\to X$ be bijections. Then $B(x,y)=(\tau(y),\sigma(x))$,
$V(x,y)=(vy,v^{-1}x)$ defines a virtual birack structure on $X$ 
iff $\sigma\tau=\tau\sigma$, $\sigma\nu=\nu\sigma$ and $\tau\nu=\nu\tau$.
\item \textit{Fundamental virtual birack of a framed virtual link.} For a 
blackboard framed virtual link diagram $L$, let $G$ be a set of generators 
corresponding bijectively with the semiarcs of $L$. The set $W(L)$ of 
\textit{virtual birack words} in $L$ is defined recursively by the rules 
\begin{itemize}
\item[(i)] $g\in G\Rightarrow g\in W(L)$ and
\item[(ii)] $g,h\in W(L)\Rightarrow B_i^{\pm 1}(g,h), S_i^{\pm 1}(g,h)\in W(L),
V_i^{\pm 1}(g,h),$ and $vS_i^{\pm 1}(g,h)\in W(L)$ where $i=1,2$.
\end{itemize}
The set of equivalence classes of $W(L)$ under the equivalence relation
generated by the \textit{crossing relations} in definition \ref{def:lab} 
and the virtual birack axioms is then a virtual birack whose isomorphism 
class is independent of the diagram chosen to represent $L$. This virtual 
birack is called the \textit{fundamental virtual birack} of $L$, denoted 
$VB(L)$.
\end{itemize}}
\end{example}

\begin{definition}\textup{
Let $X=\{x_1,\dots,x_n\}$ be a finite set with virtual birack structure maps 
$B,V$. We can conveniently specify the maps $B,V$ with a \textit{virtual 
birack matrix}
\[M_X=\left[\begin{array}{c|c|c|c} B_1 & B_2 & V_1 & V_2 \\ \end{array}\right]\]
where if $B(x_i,x_j)=(x_k,x_l)$ and $V(x_i,x_j)=(x_m,x_n)$ then 
$(B_1)_{j,i}=k$, $(B_2)_{i,j}=l$, $(V_1)_{j,i}=m$ and $(V_2)_{i,j}=n$. These 
matrices can be understood as operation tables for the operations 
$(x_j)^{(x_i)}$, $(x_i)_{(x_j)}$, $(x_j)^{\widetilde{(x_i)}}$, and
$(x_i)_{\widetilde{(x_j)}}$ respectively. Note the
reversed order of $i,j$ in $B_1$ and $V_1$; this convention is chosen so that
the rows and outputs represent the same strand while the columns represent 
the other strand crossing over, under or virtually respectively.
}\end{definition}

\begin{example}\textup{
Let $X=\mathbb{Z}_5=\{1,2,3,4,5\}$. Then the $(v,t,s,r)$-virtual birack 
structure on $X$ given by $v=2,t=3,s=4,r=3$ has virtual birack matrix
\[M_X=\left[\begin{array}{ccccc|ccccc|ccccc|ccccc}
2 & 1 & 5 & 4 & 3 & 3 & 3 & 3 & 3 & 3 & 4 & 4 & 4 & 4 & 4 & 2 & 2 & 2 & 2 & 2 \\
5 & 4 & 3 & 2 & 1 & 1 & 1 & 1 & 1 & 1 & 3 & 3 & 3 & 3 & 3 & 4 & 4 & 4 & 4 & 4 \\
3 & 2 & 1 & 5 & 4 & 4 & 4 & 4 & 4 & 4 & 2 & 2 & 2 & 2 & 2 & 1 & 1 & 1 & 1 & 1 \\
1 & 5 & 4 & 3 & 2 & 2 & 2 & 2 & 2 & 2 & 4 & 4 & 4 & 4 & 4 & 3 & 3 & 3 & 3 & 3 \\
4 & 3 & 2 & 1 & 5 & 5 & 5 & 5 & 5 & 5 & 5 & 5 & 5 & 5 & 5 & 5 & 5 & 5 & 5 & 5 \\
\end{array}\right]\]
}\end{example}

\begin{definition}\label{def:lab}\textup{
A \textit{virtual birack labeling} of an oriented blackboard framed virtual 
link diagram $L$ by a virtual birack $X$, also called an \textit{$X$-labeling} 
of $L$, is an assignment of an element of $X$ to each semiarc in $L$ such 
that the conditions
\[\includegraphics{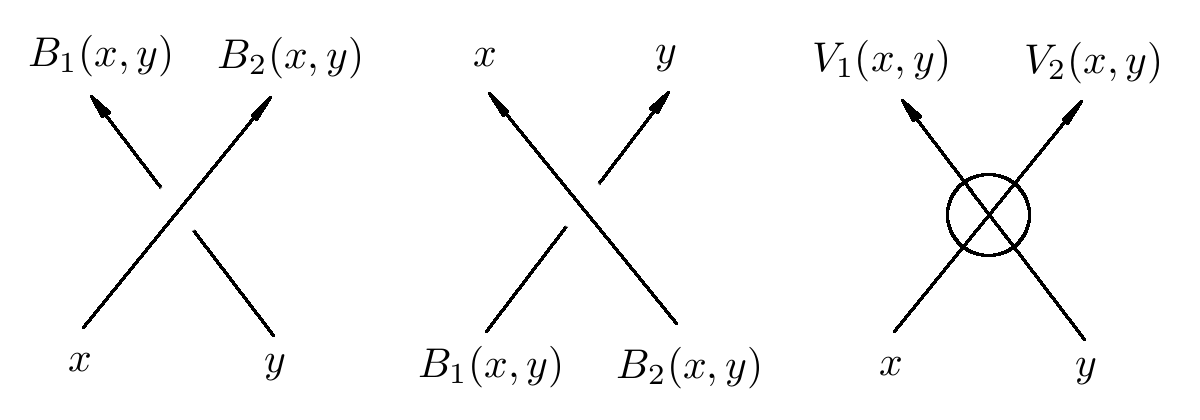}\]
are satisfied at every crossing.}
\end{definition}

The virtual birack axioms are the result of translating the the oriented 
blackboard framed virtual Reidemeister moves into conditions on labelings of 
semiarcs using the labeling rule in definition \ref{def:lab}. Invertibility
of $B,V$ and existence, uniqueness and invertibility of $S$ and $vS$ encode the
Reidemeister II and vII moves; diagonal invertibility of $B$ satisfies the
framed I move, while the requirement that $V$ fixes the diagonal satisfies 
the vI move. The Yang-Baxter equations then encode the Reidemeister
III, vIII and v moves. Indeed, by construction we have:
\begin{theorem}
\textup{If $X$ is a virtual birack and $L$ and $L'$ are virtual link diagrams 
related 
by blackboard framed virtual Reidemeister moves, then there is a bijection 
between the sets of virtual birack labelings of $L$ and $L'$.}
\end{theorem}

\begin{definition}
\textup{Let $L$ be a blackboard framed virtual link diagram and $X$ a finite 
virtual birack. The cardinality of the set of $X$-labelings of $L$, denoted 
$\Phi^B_{X}(L)$, is the \textit{basic virtual birack counting invariant} of 
$L$ with respect to $X$.}
\end{definition}

Note that $\Phi^B_X(L)=|\mathrm{Hom}(VB(L),X)|$ where $\mathrm{Hom}(VB(L),X)$
is the set of virtual birack homomorphisms from the fundamental virtual
birack $VB(L)$ of $L$ to the labeling birack $X$.

Let $X$ be a virtual birack. The maps $\alpha,\pi:X\to X$ defined by 
$\alpha=(S^{-1}\circ \Delta)_2^{-1}$ and $\pi=(S^{-1}\circ \Delta)_1\circ \alpha$
give the virtual birack labels on the semiarcs in a framed type I move:
\[\includegraphics{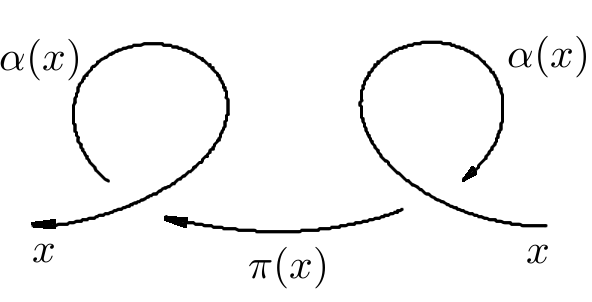}\]
Since $\pi$ represents going through a positive kink, $\pi$ is known as the
\textit{kink map} of the virtual birack $X$. The order of $\pi$, i.e. the 
smallest positive integer $N$ such that $\pi^N=\mathrm{Id}:X\to X$, is the 
\textit{birack rank} or \textit{birack characteristic} of $X$. 

Virtual birack labelings of a virtual knot or link are preserved by 
blackboard framed virtual Reidemeister moves but not in general by 
Reidemeister I moves. However, if a virtual birack $X$ has finite birack 
rank $N$, then labelings of $X$ are preserved by \textit{$N$-phone cord moves}:
\[\includegraphics{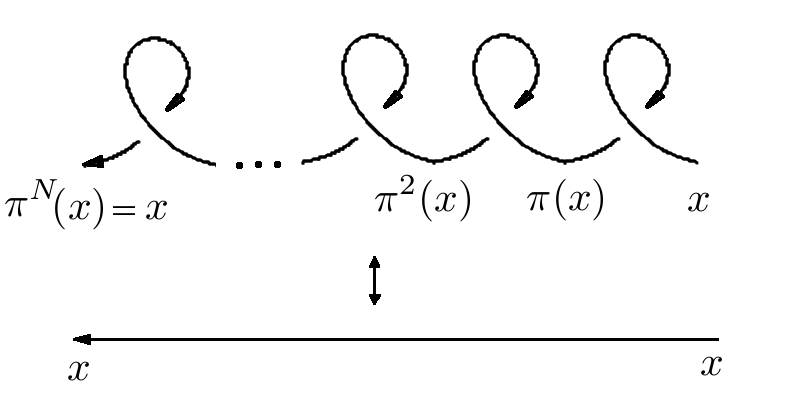}\]
In particular, if $L$ and $L'$ are related by framed virtual Reidemeister and
$N$-phone cord moves, then $\Phi^B_X(L)=\Phi^B_X(L')$. See \cite{N} for more.

Thus, if $L$ is an unframed oriented virtual link of $c$ components, there is 
a $c$-dimensional lattice of framings of $L$ corresponding to writhe vectors 
in $\mathbb{Z}^c$, each representing a distinct framed link. We therefore have, 
associated to an unframed oriented virtual link, a $\mathbb{Z}^c$-lattice of 
basic birack counting invariants. The fact that $N$-phone cord moves preserve 
$\Phi^B_X(L)$ then implies that this lattice is tiled with repeats of a 
$c$-dimensional tile corresponding to $(\mathbb{Z}_N)^c$. In particular, the 
sum over one tile of the basic counting invariants yields an invariant of 
unframed oriented virtual links known as the \textit{integral virtual birack 
counting invariant}.

\begin{definition}\textup{
Let $X$ be a finite virtual birack of birack rank $N$ and let $L$ be 
an unframed oriented virtual link of
$c$ components. For every $\mathbf{w}\in (\mathbb{Z}_N)^c$, let 
$(L,\mathbf{w})$ be a diagram of $L$ with framing vector $\mathbf{w}$. Then
\[\Phi_X^{\mathbb{Z}}(L)=\sum_{\mathbf{w}\in (\mathbb{Z}_N)^c} \Phi_X^B(L,\mathbf{w})\]
is the \textit{integral virtual birack counting invariant}.
}\end{definition}

For instance, if $N=2$ then for the virtual Hopf link we have
\[\scalebox{0.7}{\includegraphics{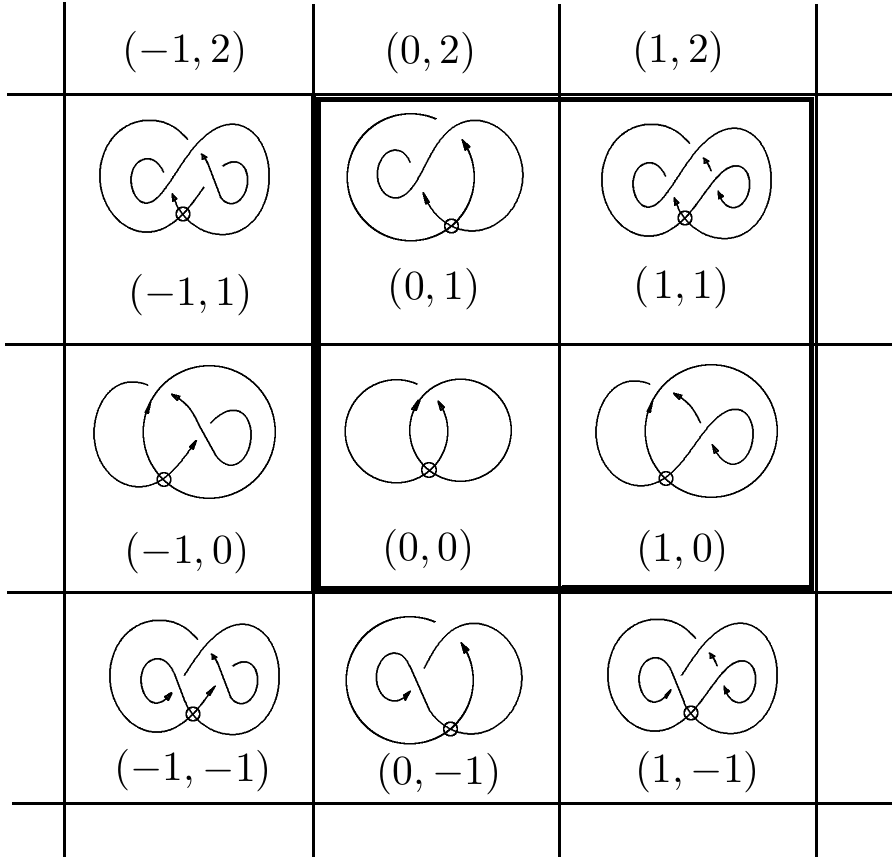}} \quad 
\raisebox{1in}{$\begin{array}{rcl} 
\Phi_X^{\mathbb{Z}}(L) & = &  \Phi_X^B(L,(0,0)) + \Phi_X^B(L,(1,0))\\
 & & +\Phi_X^B(L,(0,1))+\Phi_X^B(L,(1,1)). \\
\end{array}$}\]

By construction and theorem 1, we have
\begin{theorem}\textup{
If $X$ is a finite virtual birack and $L$ and $L'$ are virtually isotopic
oriented virtual link diagrams, then 
$\Phi_X^{\mathbb{Z}}(L)=\Phi_x^{\mathbb{Z}}(L')$.
}\end{theorem}

\begin{example}\label{ex1}\textup{
Let $X$ be the virtual birack with underlying set $\{1,2\}$ and virtual birack 
matrix
\[M_X=\left[\begin{array}{cc|cc|cc|cc}
1 & 1 & 2 & 2 & 2 & 2 & 2 & 2 \\
2 & 2 & 1 & 1 & 1 & 1 & 1 & 1 \\
\end{array}\right].\]
This matrix describes the labelings of a virtual link diagram with 
elements of $X$ by the following rules: when going through a classical 
undercrossing, semiarc labels stay the same; at classical overcrossings 
and at virtual crossings from either direction, semiarc labels switch 
from 1 to 2 or from 2 to 1. The kink map here is the transposition 
$\pi=(12)$, which has order 
$N=2$; thus, for any link $L$, we must consider a complete set 
of diagrams with writhes mod $2$ on each component. For example, the
virtual Hopf link has $c=2$ components and thus we have $N^c=2^2=4$ 
diagrams, corresponding to writhe vectors in $(\mathbb{Z}_N)^c=
(\mathbb{Z}_2)^2=\{(0,0),(1,0),(0,1),(1,1)\}$ in a complete tile of
framings mod $2$.
It is not hard to see that, for instance, the $\mathbf{w}=(0,0)$ framing 
has no valid labelings by $X$:
\[\includegraphics{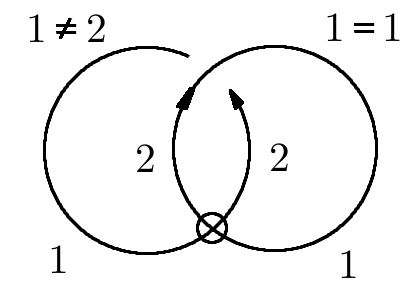}\]
It turns out that there are four valid labelings, all in the 
$\mathbf{w}=(1,0)$ framing. Thus, we have $\Phi_x^{\mathbb{Z}}(L)=4.$
\[\raisebox{-1in}{\scalebox{0.7}{\includegraphics{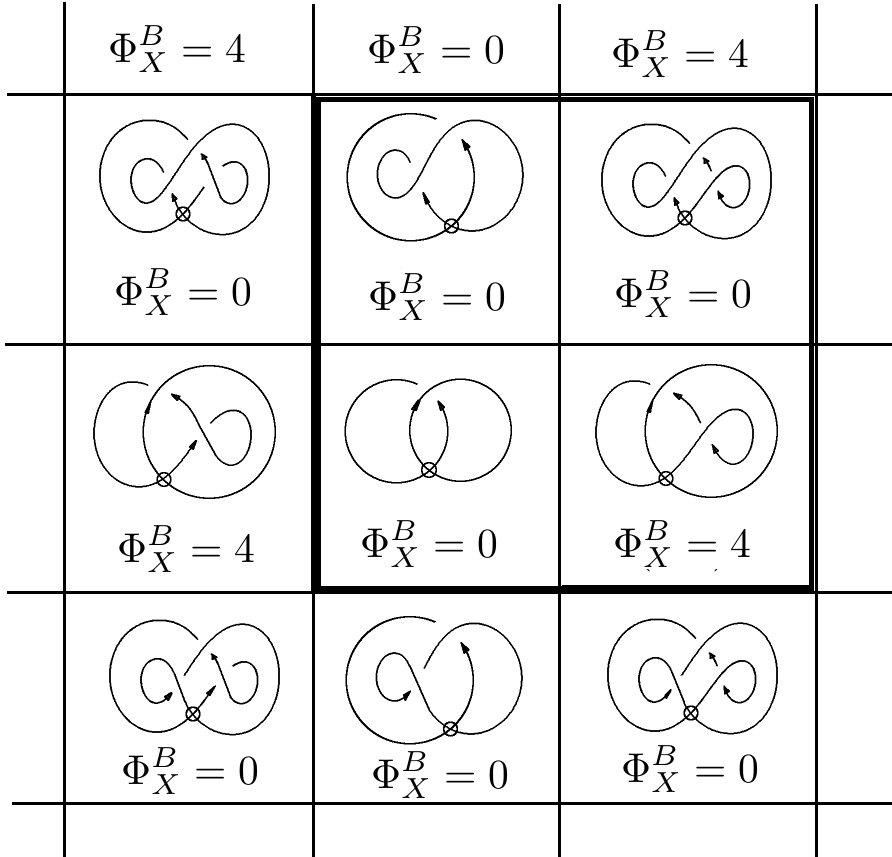}}} 
\quad \Phi_x^{\mathbb{Z}}(L)=0+4+0+0=4.\]
}\end{example}

\subsection{Virtual birack shadows}

\textit{Shadow labelings}, that is, labelings of the regions between the
arcs of a knot diagram, have been used in connection with rack, quandle,
birack and biquandle counting invariants in \cite{CJKLS,CN2,NP}. We now extend
this idea to the case of virtual biracks.

\begin{definition}
\textup{Let $X$ be a virtual birack and $S$ a set. A \textit{virtual birack 
shadow} structure or \textit{$X$-shadow} structure on $S$ is a right 
invertible right action of $X$ on $S$ (i.e. a map 
$\cdot:S\times X\rightarrow S)$ satisfying for all $x,y \in X$ and $A \in S$:}
\[(A \cdot y_x) \cdot x^y
=(A \cdot y_{\widetilde{x}}) \cdot x^{\widetilde{y}} 
= (A \cdot  x) \cdot y\]
\end{definition}

Elements of an $X$-shadow $S$ are used to label the regions between the strands
in an $X$-labeled virtual link diagram: 

\begin{definition}
\textup{Let $X$ be a virtual birack, $S$ an $X$-shadow, and $L$ a 
blackboard framed oriented virtual link diagram. Then a \textit{shadow 
labeling} or \textit{$X,S$-labeling} of $L$ is an assignment of elements of 
$X$ to the semiarcs in $L$ and elements of $S$ to regions between the semiarcs 
of $L$ such that the virtual birack labeling conditions are satisfied at every
crossing and at every semiarc we have
\[\includegraphics{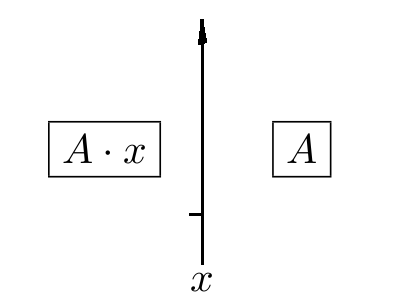}\]
}\end{definition}

The shadow axioms are chosen to guarantee that shadow labelings are 
well-defined at classical and virtual crossings. It is straightforward to show
that the requirement that shadow labels are well-defined around classical
and virtual crossings is sufficient to guarantee that shadow labelings are
preserved under Reidemeister moves. Moreover, well-definedness of shadow 
labels around positive crossings implies well-definedness around negative 
crossings.
\[\includegraphics{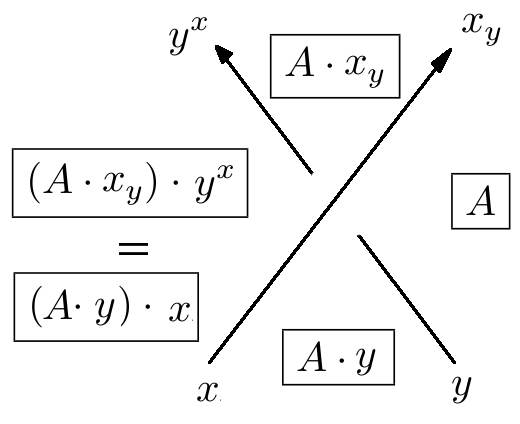}\quad \quad
\includegraphics{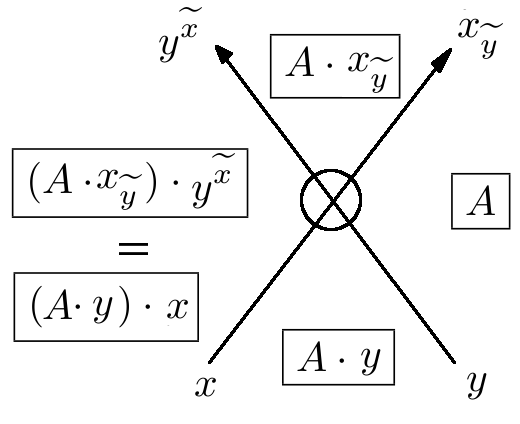}
\]

Let $X=\{x_1,\dots, x_n\}$ be a virtual birack and $S=\{A_1,\dots, A_m\}$ 
an $X$-shadow. We can specify the shadow operation with an
$m\times n$ matrix $M_S$ whose row $i$ column $j$ entry is $k$ such
that $A_i\cdot x_j=A_k$. Note that right-invertibility of $\cdot$ requires
the columns of $M_S$ to be permutations, and the columns corresponding to
birack elements $x$ and $\pi(x)$ are the same.

\begin{example}
\textup{Let $X$ be a virtual birack, $S$ any set and $\sigma:S\to S$ 
any bijection. Then $A\cdot x=\sigma(A)$ defines an $X$-shadow structure on 
$S$, called a \textit{constant action shadow}: bijectivity of $\sigma$ gives
us right-invertibility, and we have for any $x,y\in X$
\[(A \cdot y_x) \cdot x^y =(A \cdot y_{\widetilde{x}}) \cdot x^{\widetilde{y}} 
= (A \cdot  x) \cdot y=\sigma^2(x).\] 
In particular, if $S=\{A,B\}$
and every column is the transposition $(AB)$,
then $X,S$-labelings are called \textit{checkerboard colorings} of $L$.}
\end{example}

\begin{definition}
\textup{Let $X$ be a virtual birack with birack rank $N$ and $S$ an 
$X$-shadow. For any blackboard framed oriented virtual link $L$ of $c$ 
components, let $\mathcal{L}(L,X,S)$ be the set of $X,S$-labelings of $L$. 
Then set}
\[\Phi^{\mathbb{Z}}_{X,S}(L)
=\sum_{w \in (\mathbb{Z}_n)}|\mathcal{L}((L,\mathbf{w}),X,S)|.\] 
\end{definition}

It is apparent that the number of shadow labelings for a fixed $X$-labeling 
is equal to the cardinality of $S$, since a choice of shadow label for one 
``source'' region determines shadow labels for all other regions, and there
are $|S|$ such choices. Hence, we have
\begin{theorem} Let $X$ be a finite birack. For any $X$-shadow and 
oriented virtual link $L$, we have
$\Phi^{\mathbb{Z}}_{X,S}(L)= |S| \Phi^{\mathbb{Z}}_{X}(L).$
\end{theorem}

\begin{corollary} 
Let $X$ be a virtual birack and $S$ an $X$-shadow. If $L$ and $L'$ 
are virtually isotopic virtual links, then we have
$\Phi^{\mathbb{Z}}_{X,S}(L)= \Phi^{\mathbb{Z}}_{X,S} (L').$
\end{corollary}

\begin{example}\label{ex2}\textup{
Let $X$ be the virtual birack from example \ref{ex1} and let $S$ be the
checkerboard coloring shadow, i.e. the set $S=\{A,B\}$ with shadow matrix 
\[M_S=\left[\begin{array}{cc} B & B \\ A & A\end{array}\right].\] Then over a 
complete period of framings mod 2, there are eight shadow labelings of
the virtual Hopf link from example \ref{ex1}, i.e. $\Phi_{X,S}^{\mathbb{Z}}(L)=8$.
}\end{example}

\subsection{Virtual Birack Shadow Algebras and Modules} \label{vbs}

The virtual shadow counting invariant is an enhancement of the integral 
virtual  birack counting invariant, but an enhancement that is equivalent 
to the original unenhanced invariant. To get a stronger enhancement, we need 
additional algebraic structure.

\begin{definition}\label{def:vba}
\textup{Let $X$ be a set with virtual birack structure $B(x,y)=(y^x,x_y)$
and $V(x,y)=(y^{\widetilde{x}},x_{\widetilde{y}})$ and let $S$ be an $X$-shadow.
Then the \textit{virtual birack shadow algebra} $\mathbb{Z}[X,S]$ is the quotient
of the polynomial algebra
$\mathbb{Z}[v_{A,x,y}^{\pm 1},t_{A,x,y}^{\pm 1},s_{A,x,y},r_{A,x,y}^{\pm 1}]$
for $A\in S, x,y\in X$ modulo the ideal $I$ generated by elements of the form}
\[\begin{array}{lllll}
\bullet & v_{A,y_{\widetilde{x}},x^{\widetilde{y}}}-v_{A,x,y} & \quad 
\bullet & v_{A\cdot y_{\widetilde{z}},x,z^{\widetilde{y}}}v_{A,y,z}-v_{A\cdot(x_{\widetilde{y}})_{\widetilde{z}},y^{\widetilde{x}},z^{\widetilde{x_{\widetilde{y}}}}}v_{A,x,y} \\
\bullet &v_{A\cdot(x_{\widetilde{y}})_{\widetilde{z}},y^{\widetilde{x}},z^{\widetilde{x_{\widetilde{y}}}}}v_{A,x_{\widetilde{z^{\widetilde{y}}}},y_{\widetilde{z}}}-v_{A\cdot z,x,y}v_{A,y,z} & \quad 
\bullet & v_{A\cdot z,x,y}v_{A,x_{\widetilde{y}},z}-v_{A\cdot y_{\widetilde{z}},x,z^{\widetilde{y}}} v_{A,x_{\widetilde{z^{\widetilde{y}}}},y_{\widetilde{z}}} \\
\bullet &v_{A\cdot y_{z},x,z^{y}}v_{A,x_{\widetilde{z^y}},y_z}-v_{A\cdot z,x,y},v_{A,x_{\widetilde{y}},z} & \quad 
\bullet & v_{A,x_{\widetilde{z^y}},y_z}r_{A,y,z}-r_{A\cdot(x_{\widetilde{y}})_{\widetilde{z}},y^{\widetilde{x}},z^{\widetilde{x_{\widetilde{y}}}}}v_{A\cdot z,x,y} \\
\bullet & v_{A\cdot y_{z},x,z^{y}}t_{A,y,z}-t_{A\cdot(x_{\widetilde{y}})_{\widetilde{z}},y^{\widetilde{x}},z^{\widetilde{x_{\widetilde{y}}}}}v_{A,x_{\widetilde{y}},z} & \quad 
\bullet & v_{A\cdot y_{z},x,z^{y}}s_{A,y,z}-s_{A\cdot(x_{\widetilde{y}})_{\widetilde{z}},y^{\widetilde{x}},z^{\widetilde{x_{\widetilde{y}}}}}v_{A\cdot z,x,y} \\
\bullet & r_{A,x_{z^y},y_z}r_{A\cdot y_z,x,z^y}-              r_{A,x_y,z}r_{A\cdot z,x,y}& \quad 
\bullet & t_{A,x_{z^y},y_z}r_{A,y,z}-         r_{A\cdot x_{yz},y^x,z^{x_y}}t_{A\cdot z,x,y} \\
\bullet & s_{A,x_{z^y},y_z}r_{A\cdot y_z,x,z^y}-r_{A\cdot x_{yz},y^x,z^{x_y}}s_{A\cdot z,x,y}& \quad 
\bullet & t_{A\cdot y_z,x,z^y}t_{A,y,z}-t_{A\cdot x_{yz},y^x,z^{x_y}}t_{A,x_y,z} \\
\bullet & t_{A\cdot y_z,x,z^y}s_{A,y,z}-s_{A\cdot x_{yz},y^x,z^{x_y}}t_{A\cdot z,x,y} & & & \\
\bullet & \multicolumn{4}{l}{s_{A\cdot y_z,x,z^y}-        t_{A\cdot x_{yz},y^x,z^{x_y}}s_{A,x_y,z}r_{A\cdot z,x,y}
                                -s_{A\cdot x_{yz},y^x,z^{x_y}}s_{A\cdot z,x,y}, \quad
\mathrm{and}}  \\
\bullet & \multicolumn{4}{l}{1-\displaystyle{\prod_{k=0}^{N-1}
(t_{A\cdot^{-1}\alpha(\pi^{k}(x)),\pi^k(x),\alpha(\pi^k(x))}
 r_{A\cdot^{-1}\alpha(\pi^{k}(x)),\pi^k(x),\alpha(\pi^k(x))}
+s_{A\cdot^{-1}\alpha(\pi^{k}(x)),\pi^k(x),\alpha(\pi^k(x))})}} \\
\end{array}\]
\end{definition}

The virtual birack algebra is motivated by the $(v,t,s,r)$-virtual birack
definition; given an $X,S$-labeled link diagram, we define a secondary
labeling of the semiarcs by beads which obey $(v,t,s,r)$-style relations
with coefficients which depend on the $X,S$ labels at the crossing.
The relations are obtained from the framed virtual Reidemeister moves and the
$N$-phone cord move using the bead relations
\[\begin{array}{ccc}
\includegraphics{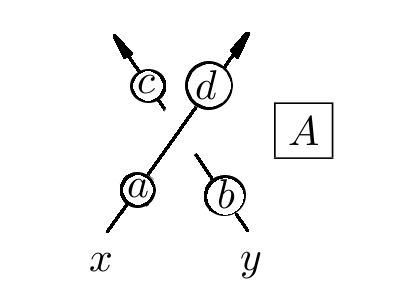} &\includegraphics{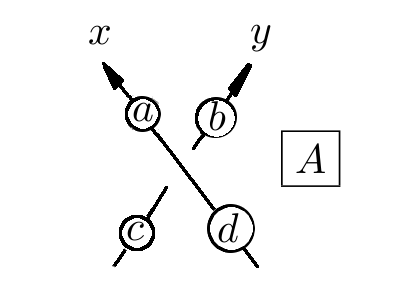} 
&\includegraphics{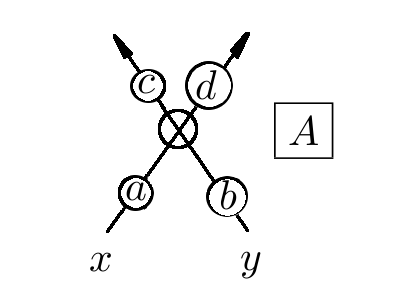} \\
\begin{array}{rcl}
c & = & t_{A,x,y} b+s_{A,x,y} a \\
d & = & r_{A,x,y} b \\
\end{array} &
\begin{array}{rcl}
c & = & t_{A,x,y} b+s_{A,x,y} a \\
d & = & r_{A,x,y} b \\
\end{array} &
\begin{array}{rcl}
c & = & v_{A,x,y} b \\
d & = & v_{A,x,y}^{-1}a. \\
\end{array} \\
\end{array}\]

For instance, the virtual move vIII yields the requirements
\[\raisebox{-1.2in}{\includegraphics{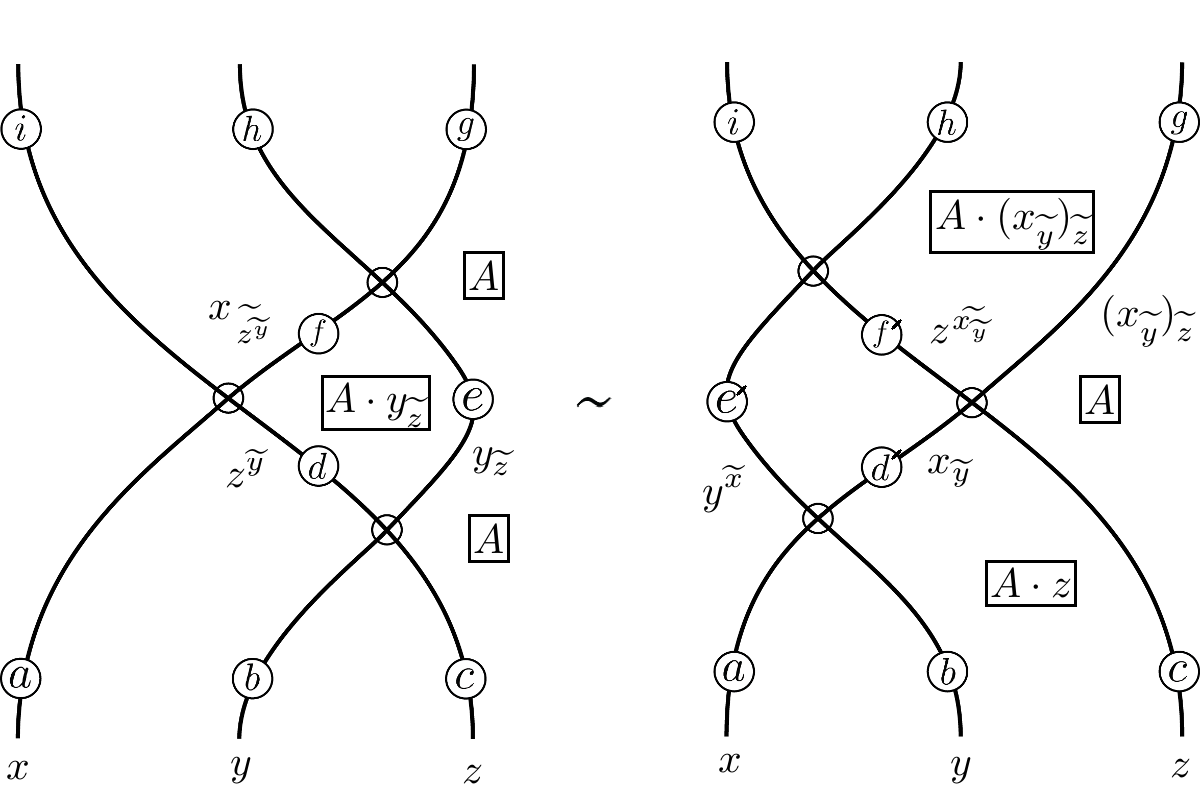}}\quad
\begin{array}{rcl}
g & = & v^{-1}_{A,x_{\widetilde{z^{\widetilde{y}}}},y_{\widetilde{z}}}
v^{-1}_{A\cdot y_{\widetilde{z}},x,z^{\widetilde{y}}} a\\
 & = & v^{-1}_{A,x_{\widetilde{y}},z}v^{-1}_{A\cdot z, x, y}a\\
& & \\
h & = & v_{A,x_{\widetilde{z^{\widetilde{y}}}},y_{\widetilde{z}}}v^{-1}_{A,y,z}b \\
 & = & v^{-1}_{A\cdot (x_{\widetilde{y}})_{\widetilde{z}},y^{\widetilde{x}},z^{\widetilde{x_{\widetilde{y}}}}}v_{A\cdot z, x,y}b \\
& & \\
i & = & v_{A,x_{\widetilde{z^{\widetilde{y}}}},y_{\widetilde{z}}}v_{A,y,z}c \\
 & = & v^{-1}_{A\cdot (x_{\widetilde{y}})_{\widetilde{z}},y^{\widetilde{x}},z^{\widetilde{x_{\widetilde{y}}}}}v_{A,x,y} c\\
\end{array}\]
Repeating for the other framed virtual Reidemeister moves yields the
relations in definition \ref{def:vba}.

\begin{definition}
\textup{A \textit{virtual shadow module} or $\mathbb{Z}[X,S]$-module is a
representation of $\mathbb{Z}[X,S]$, i.e., an
abelian group $G$ with a family of automorphisms
$v_{A,x,y},t_{A,x,y},r_{A,x,y}:G\to G$ and endomorphisms $s_{A,x,y}:G\to G$
for $A\in S, x,y\in X$ such that each generator of the ideal $I$ in definition
\ref{def:vba} is the zero map.}
\end{definition}

\begin{example}\label{ex:R}\textup{
Let $G$ be a commutative ring. For any $X$-shadow $S=\{A_1,A_2,\dots,A_m\}$ we 
can give
$G$ the structure of a $\mathbb{Z}[X,S]$-module by choosing invertible elements
$v_{A,x,y},t_{A,x,y},r_{A,x,y}\in G^{\ast}$ and elements $s_{A,x,y}\in G$ for 
each $A\in S$ and $x,y\in X$ such that the ideal $I\subset G$ in definition 
\ref{def:vba} is zero. We can express such a structure with an $m\times 4$ 
block matrix of $|X|\times |X|$ blocks
\[M_G=\left[\begin{array}{c|c|c|c}
V_{A_1} & T_{A_1} & S_{A_1} & R_{A_1} \\ \hline
V_{A_2} & T_{A_2} & S_{A_2} & R_{A_2} \\ \hline
\vdots & \vdots & \vdots & \vdots \\ \hline
V_{A_m} & T_{A_m} & S_{A_m} & R_{A_m} \\
\end{array}
\right]\]
where the row $j$ column $k$ entry of $T_{A_i}$ is $t_{A_i,x_j,y_k}$, etc.}
\end{example}

\begin{example}\label{ex3}
\textup{Let $X$ be a virtual birack and $S$ an $X$-shadow.
For any $X,S$-labeling $f$ of a blackboard framed oriented link diagram $L$, 
there is an $X,S$-module generated by beads associated to semiarcs in $L$ 
with relations determined by crossings, called the \textit{fundamental 
$X,S$-module of the labeling $f$ of $L$}, denoted $\mathbb{Z}[f]$. 
For instance, the $X,S$-labeling of the virtual Hopf link diagram below 
where $X,S$ are as in example \ref{ex1} has fundamental 
$\mathbb{Z}[X,S]$-module $\mathbb{Z}[f]$
with presentation matrix $M_{\mathbb{Z}[f]}$ below.
\[\raisebox{-0.8in}{\includegraphics{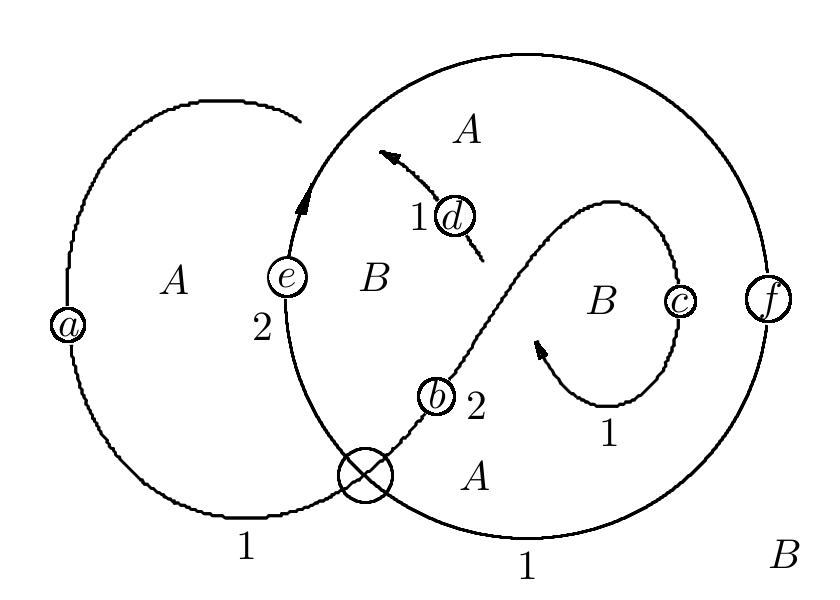}} \quad 
M_{\mathbb{Z}[f]}=
\left[\begin{array}{cccccc}
0 & s_{B,2,1} & t_{B,2,1} & -1 & 0 & 0 \\
0 & r_{B,2,1} & -1 & 0 & 0 & 0 \\
-1 & 0 & 0 & t_{A,2,1} & s_{A,2,1} & 0\\
0 & 0 & 0 & 0 & r_{A,2,1} & -1  \\
0 & 0 & 0 & 0 & -1 & v_{A,1,1}  \\
v_{A,1,1}^{-1} & -1 & 0 & 0 & 0 & 0 \\ 
\end{array}\right]
\]}\end{example}

\begin{example}
\textup{Let $X$ be the virtual birack on one element and $S$ the $X$-shadow with
one element. Then there is a unique $X,S$-labeling of any virtual link $L$, and
its fundamental $X,S$-module of virtual shadow module over the ring 
$G=\mathbb{Z}[t^{\pm 1},r^{\pm 1}]$ 
with matrix 
$M_G=\left[\begin{array}{c|c|c|c} 1 & t & 1-tr & r \end{array}\right]$
is known as the \textit{generalized Alexander module} of $L$; the determinant
of the matrix presenting $M_{\mathbb{Z}[f]}$ is a two-variable Laurent 
polynomial invariant
of virtual links known as the \textit{Generalized Alexander polynomial.} See
\cite{KR, SW} for more.
}\end{example}

\begin{definition}
\textup{Let $L$ be a virtual link of $c$ components, $X$ a virtual birack
of birack rank $N$, $S$ an $X$-shadow and $G$ an abelian group with the
structure of a $\mathbb{Z}[X,S]$-module. Let $\mathcal{L}((L,\mathbf{w}),X,S)$
be the set of $X,S$-labelings of a diagram of $L$ with writhe vector
$\mathbf{w}\in(\mathbb{Z}_N)^c$. Then the \textit{virtual shadow module
multiset} of $L$ with respect to $G$ is the multiset of $G$-modules}
\[\Phi^{M,G}_{X,S}(L)=\left\{\mathrm{Hom}_G(\mathbb{Z}[f],G) :
f\in \mathcal{L}((L,\mathbf{w}),X,S),
\mathbf{w}\in (\mathbb{Z}_N)^c\right\}\]
\textup{and the \textit{virtual shadow module polynomial} of $L$ with respect
to $G$ is}
\[\Phi^{G}_{X,S}(L)=\sum_{\mathbf{w}\in (\mathbb{Z}_N)^c}\left(
\sum_{f\in \mathcal{L}((L,\mathbf{w}),X,S)} u^{|\mathrm{Hom}_G(\mathbb{Z}[f],G)|}\right).\]
\end{definition}

By construction, we have
\begin{theorem}
If $L$ and $L'$ are virtually isotopic virtual links, then
$\Phi^{M,G}_{X,S}(L)=\Phi^{M,G}_{X,S}(L')$ and $\Phi^G_{X,S}(L)=\Phi^G_{X,S}(L')$.
\end{theorem}

\begin{example}\label{ex4}\textup{
Let $X,S$ be the virtual birack and shadow from example \ref{ex2}. Let 
$G=\mathbb{Z}_5$; then thinking of $G$ as a commutative ring as in 
example \ref{ex:R}, we can give $G$ the structure of a 
$\mathbb{Z}[X,S]$-module with the shadow module matrix
\[M_G=\left[\begin{array}{cc|cc|cc|cc}
2 & 2 & 1 & 4 & 3 & 4 & 1 & 1 \\ 
2 & 2 & 1 & 4 & 1 & 3 & 1 & 1 \\ \hline
1 & 1 & 2 & 3 & 3 & 1 & 2 & 2 \\
1 & 1 & 2 & 3 & 4 & 3 & 2 & 2
\end{array}\right].\]
To find the bead labelings of the $X,S$-labeling of the virtual Hopf link from 
example \ref{ex3}, we replace the entries in the matrix $M_{\mathbb{Z}[f]}$ 
presenting the
fundamental $\mathbb{Z}[X,S]$-module of $f$ with their values in $M_G$ to 
obtain a matrix over $\mathbb{Z}_5$ whose solution space yields the set of 
bead labelings of the semiarcs with beads in $\mathbb{Z}_5$:
\[\left[\begin{array}{cccccc}
0 & 4 & 2 & 4 & 0 & 0 \\
0 & 2 & 4 & 0 & 0 & 0 \\
4 & 0 & 0 & 1 & 1 & 0 \\
0 & 0 & 0 & 0 & 1 & 4 \\
0 & 0 & 0 & 0 & 4 & 2 \\
3 & 4 & 0 & 0 & 0 & 0 \\
\end{array}\right]
\longrightarrow
\left[\begin{array}{cccccc}
1 & 0 & 0 & 0 & 0 & 0 \\
0 & 1 & 0 & 0 & 0 & 0 \\
0 & 0 & 1 & 0 & 0 & 0 \\
0 & 0 & 0 & 1 & 0 & 0 \\
0 & 0 & 0 & 0 & 1 & 0 \\
0 & 0 & 0 & 0 & 0 & 1 \\
\end{array}\right]
\]
Thus, $X,S$-labeling in example \ref{ex2} has only the all-zeroes bead
labeling with respect to our chosen $M_R$, and this labeling contributes
$u^1$ to the value of $\Phi_{X,S}^G(L)$. Repeating for all eight shadow
labelings, we get $\Phi_{X,S}^G(L)=8u$. 
}\end{example}

\begin{example}
\textup{Repeating the computation from example \ref{ex4} for the unlink of
two components and the classical Hopf link with the same values of $X,S$ and 
$M_G$, we obtain
\[\begin{array}{ccc}
\includegraphics{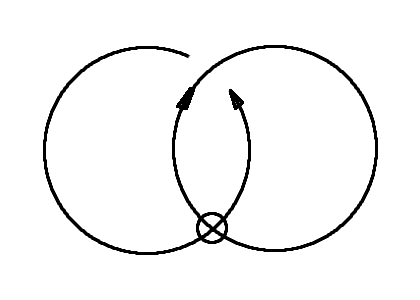} & 
\includegraphics{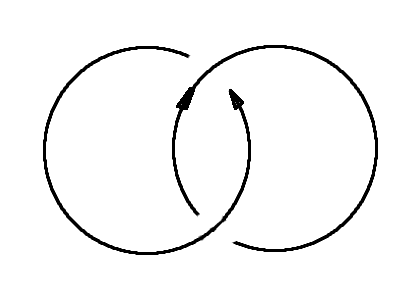} & 
\includegraphics{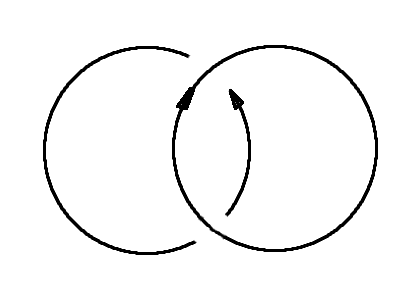} \\
\Phi_{X,S}^G(VH)=8u & \Phi_{X,S}^G(L2a1)=8u^5 & \Phi_{X,S}^G(U_2)=8u^{25} \\
\end{array}\]
In particular, this example shows that $\Phi_{X,S}^G$ is not determined by
$\Phi_{X,S}^{\mathbb{Z}}$, since all three links have $\Phi_{X,S}^{\mathbb{Z}}=8.$
}\end{example}

\section{\large\textbf{Twisted Virtual Shadow Algebras and Modules}}\label{tvbs}

\subsection{Twisted virtual knots and links}

Twisted virtual knots and links were introduced in \cite{B} in 2008, 
generalizing the notion of virtual knots as knots in $I$-bundles over 
compact orientable surfaces to the case of compact nonorientable surfaces.
Classical crossings in a twisted virtual link diagram correspond to
crossings in $\Sigma\times I$, while virtual crossings correspond to
genus in $\Sigma\times I$ and \textit{twist bars} on strands indicate
when a strand has traversed a cross-cap in $\Sigma$.

\[\raisebox{0.2in}{\includegraphics{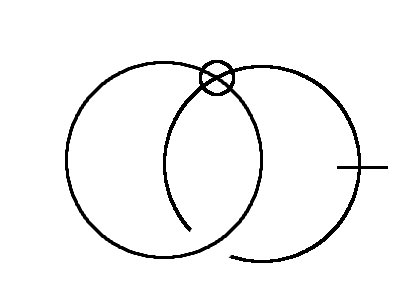}}\quad 
\includegraphics{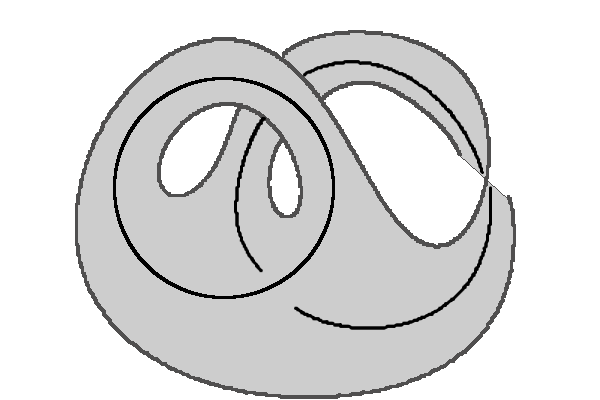}\]

A twisted oriented blackboard framed virtual link is an equivalence class of 
twisted oriented virtual link diagrams, i.e. oriented virtual link diagrams with
twist bars, under the equivalence relation generated by the virtual Reidemeister
moves together with the \textit{twisted moves}:
\[\includegraphics{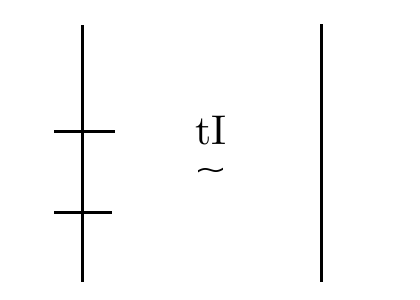}\quad\includegraphics{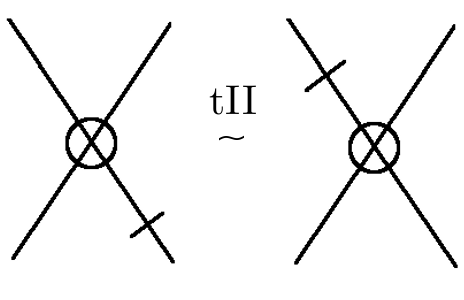}
\quad \includegraphics{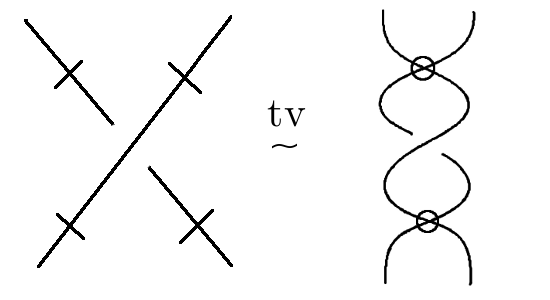}\]

\begin{remark}\textup{
Note that while a twist bar can always move past a virtual crossing, in
general it cannot move past a classical crossing. A twist bar \textit{can} 
be pushed through a classical kink, so the blackboard framing given by 
writhe numbers for each component is still well defined in the twisted 
virtual setting. See \cite{B,CN6} for more.
}\end{remark}

\subsection{Twisted virtual biracks}

Adding twists to our set of allowed moves adds a map $T:X\to X$ to our algebraic
structure defined by semiarc labelings. Specifically, in a twisted virtual
link diagram, semiarcs in a twisted virtual link diagram are obtained by 
dividing the link at classical over and undercrossings, virtual crossings 
and twist bars. As before, we will refer to a portion of a twisted virtual 
link diagram between classical crossing points as a 
\textit{classical semiarc}. Our semiarc labeling rule now includes:
\[\includegraphics{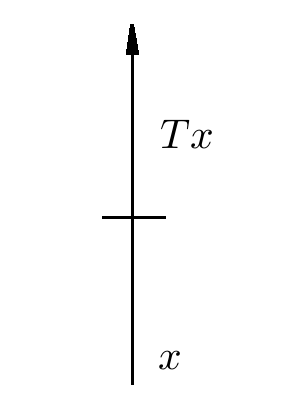}\]

\begin{definition}
\textup{Let $X$ be a set and let $\Delta : X \rightarrow X \times X$ be the 
diagonal map $\Delta(x) = (x,x)$. A \textit{twisted virtual birack} structure
on $X$ is a pair of invertible maps $B,V:X\times X \rightarrow X\times X$ and 
an involution $T : X\rightarrow X$ satisfying
\begin{itemize}
\item[(i)] $X$ is a virtual birack with operations $B$ and $V$,
\item[(ii)]
\[(T\times \mathrm{Id}_X)V = V(\mathrm{Id}_X\times T)\quad \mathrm{and} \quad 
(\mathrm{Id}_X\times T)V=V(T\times \mathrm{Id}_X), \quad \mathrm{and}\]
\item[(iii)]
\[(T\times T)B(T\times T) = VBV.\]
\end{itemize}
If we also have $(S\circ \Delta)_1 = (S\circ \Delta)_2$, $X$ is a 
\textit{twisted virtual biquandle}. See \cite{CN6,SPC} for more.}
\end{definition} 

\begin{example}\textup{
Let $\bar{\Lambda}=\mathbb{Z}[t^{\pm 1}, r^{\pm 1}, v^{\pm 1}, T]/(1-T^2,t-v^2r)$.
Then any $\bar{\Lambda}$-module is a twisted virtual birack under the 
operations $B(x,y)=(ty,rx)$, $V(x,y)=(vy,v^{-1}x)$, $T(x)=Tx$. See \cite{CN6}
for more.
}\end{example}

As with virtual biracks, we can represent a twisted virtual birack structure
on a set $X=\{x_1,\dots, x_n\}$ with a matrix (in this case, $n\times (4n+1)$) 
encoding the operation tables of the $x^y, x_y, x^{\widetilde{y}},
x_{\widetilde{y}}$ and $T(x)$ operations:
\[M_X=\left[\begin{array}{c|c|c|c|c}
B_1 & B_2 & V_1 & V_2 & T \\
\end{array}\right]\]
with $(B_1)_{j,i}=k$, $(B_2)_{i,j}=l$, $(V_1)_{j,i}=m$, $(V_2)_{i,j}=p$ and $T_i=q$
such that $B(x_i,x_j)=(x_k,x_l)$, $V(x_i,x_j)=(x_m,x_p)$ and $T(x_i)=x_q$.

A twisted virtual birack is a virtual birack; more precisely, there is
a forgetful functor from the category of twisted virtual biracks to the
category of virtual biracks defined by forgetting the twist operation, i.e.,
\[\mathcal{F}:\mathbf{Tvb}\to \mathbf{Vb} \quad \mathrm{by} \quad
\mathcal{F}(X,B,V,T)=(X,B,V).\] 
However, not every virtual birack has a compatible twisted structure 
$T:X\to X$: 

\begin{theorem}
The virtual birack $X$ with virtual birack matrix
\[M_X=\left[\begin{array}{cc|cc|cc|cc}
1 & 1 & 2 & 2 & 2 & 2 & 2 & 2 \\
2 & 2 & 1 & 1 & 1 & 1 & 1 & 1 \\
\end{array}\right]\]
has no twisted structure $T:X\to X$ satisfying all of the twisted
virtual birack axioms.\end{theorem}

\begin{proof}
A twisted structure $T:X\to X$ is a transposition, so there are only
two possible twist maps: $T=\mathrm{Id}_X$ and $T=(12)$. Then 
for both $T=\mathrm{Id}_X$ and $T=(12)$, we have 
\[(T\times T)B(T\times T)(x,y)=(y,(12)x)\ne ((12)y,x)=VBV(x,y).\]
\end{proof}


\subsection{Twisted virtual birack shadows}

\begin{definition}\textup{
Let $X$ be a twisted virtual birack. A \textit{twisted virtual birack shadow}
is a virtual birack shadow $S$ over $X$ considered as a virtual birack 
such that for every $x\in X$ we have $A\cdot Tx=A\cdot x$.
\[\includegraphics{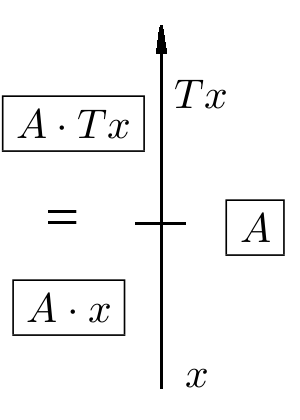}\]
}\end{definition}

As before, for any twisted virtual birack $X$ of rack rank $N$ and $X$-shadow
$S$, we have an invariant of twisted virtual links
\[\Phi^{\mathbb{Z}}_{X,S}(L)= 
\sum_{w \in (\mathbb{Z}_n)}|\mathcal{L}((L,\mathbf{w}),X,S)|\]
where $\mathcal{L}((L,\mathbf{w}),X,S)$ is the set of $X,S$-labelings of a
diagram of $L$ with writhe vector $\mathbf{w}$. Moreover, we have

\begin{theorem} Let $X$ be a twisted virtual birack and $S$ an $X$-shadow. 
Then for any unframed oriented twisted virtual link $L$,
$\Phi^{\mathbb{Z}}_{X,S}(L)= |S| \Phi^{\mathbb{Z}}_{X}(L).$
\end{theorem}

\begin{corollary} 
Let $X$ be a virtual birack and $S$ an $X$-shadow. If $L$ and $L'$ 
are virtually isotopic virtual links, then we have
$\Phi^{\mathbb{Z}}_{X,S}(L)= \Phi^{\mathbb{Z}}_{X,S} (L').$
\end{corollary}

\subsection{Twisted virtual shadow algebras}

As in the virtual birack case, we use a secondary labeling by beads 
to enhance $\Phi^{\mathbb{Z}}_{X,S}(L)$.

\begin{definition}\label{def:tvbsalg}
\textup{Let $X$ be a set with twisted virtual birack structure 
$B(x,y)=(y^x,x_y)$, $V(x,y)=(y^{\widetilde{x}},x_{\widetilde{y}})$ and $T:X\to X$
and let $S$ be an $X$-shadow. Then the \textit{twisted virtual shadow algebra} 
$\mathbb{Z}[X,S]$ is the quotient of the polynomial algebra
$\mathbb{Z}[v_{A,x,y}^{\pm 1},t_{A,x,y}^{\pm 1},r_{A,x,y}^{\pm 1}, q_{A,x}]$
for $A\in S, x,y\in X$ modulo the ideal $I$ generated by elements of the form}
\[\begin{array}{llll}
\bullet & v_{A,y_{\widetilde{x}},x^{\widetilde{y}}}-v_{A,x,y} & \quad 
\bullet & v_{A\cdot y_{\widetilde{z}},x,z^{\widetilde{y}}}v_{A,y,z}-v_{A\cdot(x_{\widetilde{y}})_{\widetilde{z}},y^{\widetilde{x}},z^{\widetilde{x_{\widetilde{y}}}}}v_{A,x,y} \\
\bullet & v_{A\cdot(x_{\widetilde{y}})_{\widetilde{z}},y^{\widetilde{x}},z^{\widetilde{x_{\widetilde{y}}}}}v_{A,x_{\widetilde{z^{\widetilde{y}}}},y_{\widetilde{z}}}-v_{A\cdot z,x,y}v_{A,y,z} & \quad 
\bullet & v_{A\cdot z,x,y}v_{A,x_{\widetilde{y}},z}-v_{A\cdot y_{\widetilde{z}},x,z^{\widetilde{y}}} v_{A,x_{\widetilde{z^{\widetilde{y}}}},y_{\widetilde{z}}} \\
\bullet & v_{A\cdot y_{z},x,z^{y}}v_{A,x_{\widetilde{z^y}},y_z}-v_{A\cdot z,x,y},v_{A,x_{\widetilde{y}},z} & \quad 
\bullet & v_{A,x_{\widetilde{z^y}},y_z}r_{A,y,z}-r_{A\cdot(x_{\widetilde{y}})_{\widetilde{z}},y^{\widetilde{x}},z^{\widetilde{x_{\widetilde{y}}}}}v_{A\cdot z,x,y} \\
\bullet & v_{A\cdot y_{z},x,z^{y}}t_{A,y,z}-t_{A\cdot(x_{\widetilde{y}})_{\widetilde{z}},y^{\widetilde{x}},z^{\widetilde{x_{\widetilde{y}}}}}v_{A,x_{\widetilde{y}},z} & \quad   
\bullet & r_{A,x_{z^y},y_z}r_{A\cdot y_z,x,z^y}- r_{A,x_y,z}r_{A\cdot z,x,y}\\ 
\bullet & t_{A,x_{z^y},y_z}r_{A,y,z}- r_{A\cdot x_{yz},y^x,z^{x_y}}t_{A\cdot z,x,y} & \quad
\bullet & t_{A\cdot y_z,x,z^y}t_{A,y,z}-t_{A\cdot x_{yz},y^x,z^{x_y}}t_{A,x_y,z} \\
\bullet & 1-q_{A,Tx}q_{A,x} & \quad 
\bullet & v_{A,x,Ty}-v_{A,x,y}  \\
\bullet & v_{A,Tx,y}-v_{A,x,y} & \quad
\bullet & q_{A\cdot y,x}v_{A,x,y}-v_{A,Tx,y}q_{A,x_{\widetilde{y}}} \\
\bullet & v_{A,x,Ty} q_{A,y} - q_{A\cdot x_{\widetilde{y}},y^{\widetilde{x}}} v_{A,x,y} 
& \quad 
\bullet & v_{A,(x_{\widetilde{y}})^{(y_{\widetilde{x}})}}q_{A,(Tx)_{(Ty)}} r_{A,T_x,T_y} q_{A\cdot y,x}v_{A,x,y}- t_{A,y^{\widetilde{x}},x_{\widetilde{y}}}  \\
\bullet & \multicolumn{3}{l}{q_{A\cdot (Tx)_{(Ty)},(Ty)^{(Tx)}}t_{A,Tx,Ty}q_{A,y}
-v_{A,(x_{\widetilde{y}})^{(y_{\widetilde{x}})}}r_{A,y^{\widetilde{x}},x_{\widetilde{y}}}v_{A,x,y} 
\quad \mathrm{and}}   \\
\bullet & \multicolumn{3}{l}{1-\displaystyle{\prod_{k=0}^{N-1}
t_{A\cdot^{-1}\alpha(\pi^{k}(x)),\pi^k(x),\alpha(\pi^k(x))}
r_{A\cdot^{-1}\alpha(\pi^{k}(x)),\pi^k(x),\alpha(\pi^k(x))}}} \\
\end{array}\]
\end{definition}

As in the virtual birack case, the twisted virtual shadow algebra relations
come from the oriented framed twisted virtual Reidemeister moves where
we interpret the generators $v_{A,x,y}, t_{A,x,y}, r_{A,x,y}$ and $q_{A,x}$ as
coefficients for beads on semiarcs as pictured: 

\[\raisebox{-0.5in}{\includegraphics{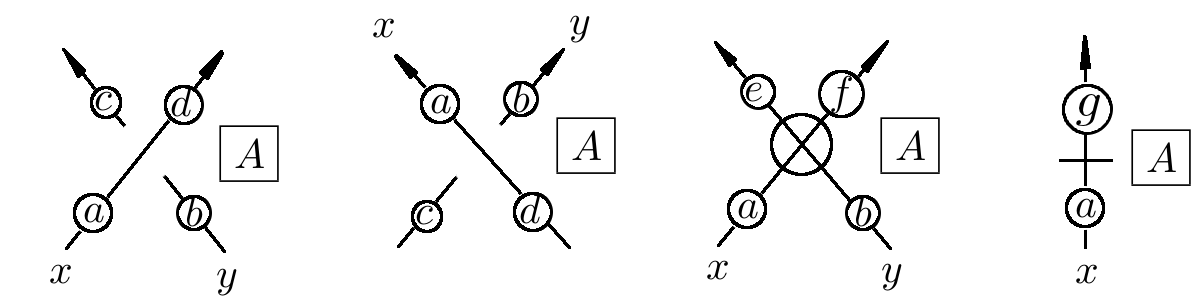} }\quad
\begin{array}{rcl}
c & = & t_{A,x,y}b \\
d & = & r_{A,x,y}a \\
e & = & v_{A,x,y}b \\
f & = & v_{A,x,y}^{-1} a \\
g & = & q_{A,x} a\\
\end{array}\]

Then for instance, move tv requires that
\[\includegraphics{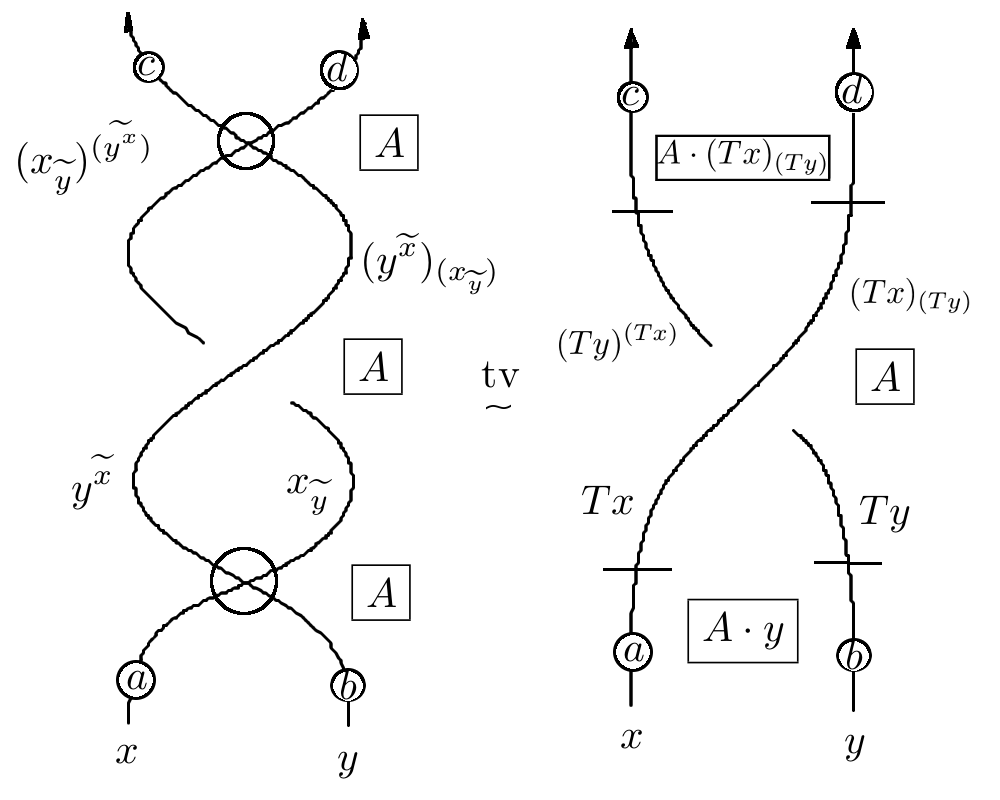} \quad
\raisebox{1.4in}{$\begin{array}{rcl}
c & = & v_{A,(x_{\widetilde{y}})^{(y_{\widetilde{x}})}}r_{A,y^{\widetilde{x}},x_{\widetilde{y}}}v_{A,x,y}b \\
& = &q_{A\cdot (Tx)_{(Ty)},(Ty)^{(Tx)}}t_{A,Tx,Ty}q_{A,y}b \\
d  & = & v_{A,(x_{\widetilde{y}})^{(y_{\widetilde{x}})}}^{-1}t_{A,y^{\widetilde{x}},x_{\widetilde{y}}}v_{A,x,y}^{-1} a\\
& = & q_{A,(Tx)_{(Ty)}} r_{A,T_x,T_y} q_{A\cdot y,x}a \\
\end{array}$}\]

\begin{remark}
\textup{If $s_{A,x,y}$ coefficients are included in the bead operations at
classical crossings as in the virtual birack algebra case, the 
tv requires that $s_{A,x,y}=0$ for all $A\in S, x,y\in X$.}
\end{remark}

\begin{definition}
\textup{Let $X$ be a twisted virtual birack and $S$ an $X$-shadow. A
\textit{twisted virtual birack shadow module} or $\mathbb{Z}[X,S]$-module
is a representation of $\mathbb{Z}[X,S]$, i.e., an abelian group $G$
with automorphisms $v_{A,x,y}, t_{A,x,y}, r_{A,x,y}, q_{A,x}$ for $A\in S$,
$x,y\in X$ such that the maps generating the ideal $I$ in definition
\ref{def:tvbsalg} are zero.}
\end{definition}

As before, an $X,S$-labeling $f$ of an oriented twisted virtual link diagram
defines a fundamental $\mathbb{Z}[X,S]$-module $\mathbb{Z}[f]$ with 
presentation matrix $M_f$ expressing the system of linear equations determined
beads on the semiarcs.

\begin{definition}
\textup{Let $L$ be a virtual link of $c$ components, $X$ a twisted virtual 
birack of birack rank $N$, $S$ an $X$-shadow and $G$ an abelian group with the
structure of a $\mathbb{Z}[X,S]$-module. Let $\mathcal{L}((L,\mathbf{w}),X,S)$
be the set of $X,S$-labelings of a diagram of $L$ with writhe vector
$\mathbf{w}\in(\mathbb{Z}_N)^c$. Then the \textit{twisted virtual shadow module
multiset} of $L$ with respect to $G$ is the multiset of $G$-modules}
\[\Phi^{M,G}_{X,S}(L)=\left\{\mathrm{Hom}_G(\mathbb{Z}[f],G) :
f\in \mathcal{L}((L,\mathbf{w}),X,S),
\mathbf{w}\in (\mathbb{Z}_N)^c\right\}\]
\textup{and the \textit{twisted virtual shadow module polynomial} of $L$ with 
respect to $G$ is}
\[\Phi^{G}_{X,S}(L)=\sum_{\mathbf{w}\in (\mathbb{Z}_N)^c}\left(
\sum_{f\in \mathcal{L}((L,\mathbf{w}),X,S)} u^{|\mathrm{Hom}_G(\mathbb{Z}[f],G)|}\right).\]
\end{definition}

By construction, we have
\begin{theorem}
If $L$ and $L'$ are twisted virtually isotopic oriented twisted virtual links, 
then
$\Phi^{M,G}_{X,S}(L)=\Phi^{M,G}_{X,S}(L')$ and $\Phi^G_{X,S}(L)=\Phi^G_{X,S}(L')$.
\end{theorem}

\section{\large\textbf{Examples and Applications}}\label{cex}

In this section we collect some examples and applications of the new
enhanced invariants.

\begin{example}
\textup{For our first application, let $X,S$ be the virtual birack and 
shadow from example \ref{ex2} and consider the $\mathbb{Z}[X,S]$-module 
structure on $G=\mathbb{Z}_5$ given by the virtual birack shadow module matrix
\[M_G=\left[\begin{array}{cc|cc|cc|cc}
1 & 1 & 1 & 1 & 4 & 1 & 3 & 3 \\
1 & 1 & 1 & 1 & 1 & 4 & 3 & 3 \\ \hline
2 & 2 & 2 & 2 & 4 & 1 & 4 & 4 \\
2 & 2 & 2 & 2 & 1 & 4 & 4 & 4 \\
\end{array}\right].\]
The oriented virtual knot $4.4$ has $\Phi_{X,S}^G(4.4)=4u$ while its  
reverse $\overline{4.4}$ has $\Phi_{X,S}^G(\overline{4.4})=4u^5$. In 
particular, since reversing the orientation of a virtual knot yields 
an isomorphic knot group, the invariants $\Phi_{X,S}^G$ are not determined 
by the isomorphism type of the knot group.
\[\begin{array}{cc}
\includegraphics{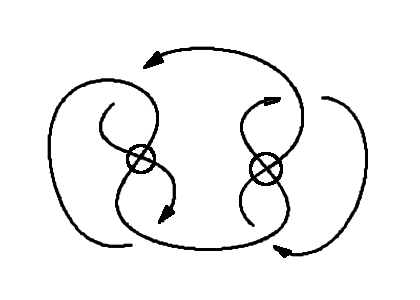} & \includegraphics{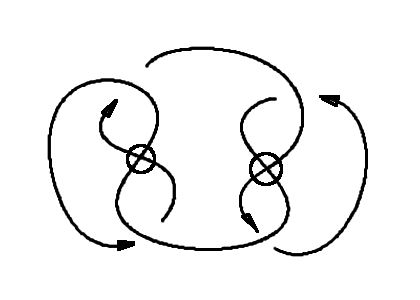} \\
\Phi_{X,S}^G(4.4)=4u & \Phi_{X,S}^G(\overline{4.4})=4u^5 \\
\end{array}
\]
}\end{example}

\begin{example}\textup{
\textit{Slavik's knot} is not detected by the arrow polynomial, and 
the \textit{Miyazawa} knot is not detected by the Miyazawa polynomial 
(see \cite{DK}). However, both are distinguished from the unknot and from each
other by $\Phi_{X,S}^G$ where $X,S$ are as in example \ref{ex3}, 
$G=\mathbb{Z}_5$ and the $\mathbb{Z}[X,S]$-module structure on $G$ is given
by the matrix:
\[M_G=\left[\begin{array}{cc|cc|cc|cc}
1 & 1 & 2 & 2 & 3 & 3 & 4 & 4 \\
1 & 1 & 2 & 2 & 3 & 3 & 4 & 4 \\ \hline
2 & 2 & 4 & 4 & 3 & 3 & 2 & 2 \\
2 & 2 & 4 & 4 & 3 & 3 & 2 & 2 \\
\end{array}\right]\]
\[\begin{array}{ccc}
\includegraphics{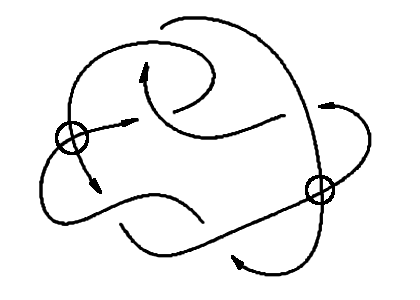} & \includegraphics{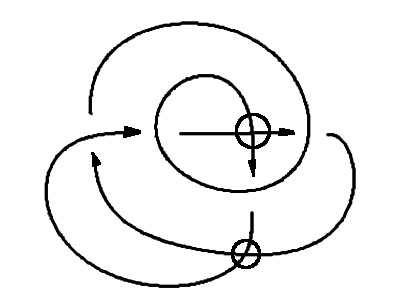} 
& \includegraphics{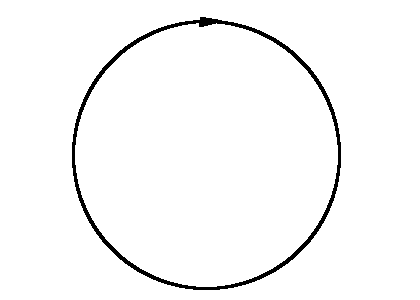} \\
\Phi_{X,S}^G(\mathrm{Slavik})=4u^{25} & \Phi_{X,S}^G(\mathrm{Miyazawa})=4u 
& \Phi_{X,S}^G(\mathrm{Unknot})=4u^5 \\
\end{array}\]
}\end{example}

\begin{example}\textup{
For our next example, we randomly a selected $\mathbb{Z}[X,S]$-module 
structure on $G=\mathbb{Z}_5$ for the virtual birack $X$ and $X$-shadow 
$S$ from example \ref{ex2} and  computed $\Phi_{X,S}^G$ for the virtual 
knots in the knot atlas \cite{KA} using our custom \texttt{python} code, 
available at \texttt{www.esotericka.org}.
\[G=\mathbb{Z}_5, \quad M_G=\left[\begin{array}{cc|cc|cc|cc}
2 & 2 & 2 & 3 & 4 & 2 & 2 & 2 \\
2 & 2 & 2 & 3 & 3 & 4 & 2 & 2 \\ \hline
4 & 4 & 1 & 4 & 4 & 3 & 1 & 1 \\
4 & 4 & 1 & 4 & 2 & 4 & 1 & 1 \\
\end{array}\right]\]
\[\begin{array}{r|l}
\Phi_{X,S}^G(L) & L \\ \hline
4u & 2.1, 3.2, 3.3, 3.4, 4.1, 4.2, 4.3, 4.4, 4.5, 4.6, 4.9, 4.11, 4.12, 4.13, 4.14,
4.15, 4.17, 4.18, 4.19, 4.20, 4.21, 4.22, \\
& 4.23, 4.24, 4.25, 4.26, 4.27, 4.28, 4.29, 4.31, 4.33, 4.34, 4.35, 4.36, 4.37, 4.38, 4.39, 4.40, 4.42, 4.43,  \\ 
& 4.44, 4.45, 4.46, 4.48, 4.49,4.51, 4.52, 4.54, 4.57, 4.60, 4.61, 4.62, 4.63, 4.64, 4.65, 4.66, 4.67, 4.69, 4.73, 4.78, \\  & 4.79, 4.80, 4.81, 4.82, 4.83, 4.84, 4.87, 4.88, 4.89, 4.92, 4.93, 4.94, 4.95, 4.97, 4.101, 4.103, 4.104 \\
4u^5 & 3.1, 3.5, 3.6, 3.7, 4.7, 4.8, 4.10, 4.16, 4.30, 4.32, 4.41, 4.47, 4.50, 4.53, 4.55, 4.56, 4.58, 4.59, 4.68, \\ 
& 4.70, 4.72, 4.74, 4.75, 4.76, 4.77, 4.85, 4.86, 4.90, 4.91, 4.96, 4.98 4.100, 4.102, 4.106, 4.107, 4.108 \\
4u^{25} & 4.71, 4.99, 4.105 \\
\end{array}\]
}\end{example}

\begin{example}\textup{For our final example, we demonstrate that
the twisted virtual shadow module invariant $\Phi_{XS}^G$ is not 
determined by the twisted Jones polynomial defined in \cite{B}. The 
twisted virtual links below both have twisted Jones polynomial
$(-A^{-2}-A^{-4})(-A^{-2}-A^2)$ but are distinguished by 
$\Phi_{XS}^G$ with the twisted virtual birack $X$, trivial $X$-shadow 
structure on $S=\{A\}$ and
$\mathbb{Z}[X,S]$-module structure on $G=\mathbb{Z}_3$ below:
\[M_X=\left[\begin{array}{ccc|ccc|ccc|ccc|c}
2 & 2 & 2 & 2 & 2 & 2 & 1 & 1 & 2 & 1 & 1 & 2 & 2 \\
1 & 1 & 1 & 1 & 1 & 1 & 2 & 2 & 1 & 2 & 2 & 1 & 1 \\
3 & 3 & 3 & 3 & 3 & 3 & 3 & 3 & 3 & 3 & 3 & 3 & 3 \\
\end{array}\right],\quad
M_S=\left[\begin{array}{ccc}
A & A & A \\
\end{array}\right]\]
\[M_G=\left[\begin{array}{ccc|ccc|ccc|c}
1 & 1 & 1 & 2 & 1 & 1 & 2 & 1 & 1 & 1 \\
1 & 1 & 1 & 1 & 2 & 1 & 1 & 2 & 1 & 1 \\
1 & 1 & 1 & 1 & 1 & 2 & 1 & 1 & 2 & 2 \\
\end{array}\right].\]
\[\begin{array}{cc}
\includegraphics{bklns-12.png} & \includegraphics{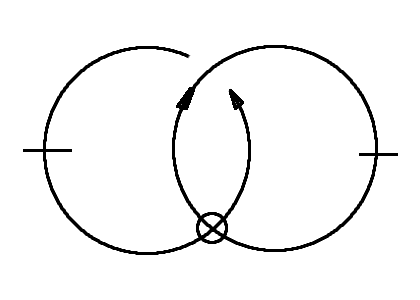} \\
\Phi_{XS}^G(L)=u+4u^9 & \Phi_{XS}^G(L')=2u+3u^9 
\end{array}\]
}\end{example}

\section{\large\textbf{Questions}}\label{q}

For simplicity, we have limited ourselves to computing $\Phi_{X,S}^G$ in 
the case when $G$ is a commutative ring -- specifically, the case where
$G=\mathbb{Z}_n$. We expect that $\Phi_{X,S}^G$ should be even stronger
if we expand to the case of noncommutative rings, e.g. $n\times n$ matrices
over $\mathbb{Z}_n$.

In the biquandle case, we can define $X$-labelings of semisheets in 
abstract knotted surface diagrams. What new or different relations, if any,
are imposed by the Roseman moves on the set of bead labelings? That is, define
virtual surface biquandle algebras and modules.

How do the enhancement strategies of the quandle counting invariant from
\cite{CEGS} apply in the case of virtual and twisted virtual shadow algebras?

What is the relationship of $\Phi_{X,S}^G$ to birack cocycle invariants?

\noindent\textsc{Department of Mathematics \\
Claremont McKenna College \\
850 Columbia Ave. \\
Claremont, CA 91711}

\end{document}